\documentclass[12pt]{amsart}
\usepackage{amstext,amsthm,amsmath,amsfonts,amssymb,amscd,latexsym,txfonts,stmaryrd}
\usepackage{tikz}
\usepackage{color}
\usepackage{epsfig}
\usepackage{cite}
\usepackage{enumitem}
\usepackage[enableskew]{youngtab}
\usepackage[colorlinks=true, pdfstartview=FitV, linkcolor=blue, citecolor=blue, urlcolor=blue]{hyperref}
\usepackage[nameinlink]{cleveref}
\usepackage{mathdots}
\usepackage[all,cmtip]{xy}

\newcommand{\flo}[1]{\left\lfloor\frac{#1}{2}\right\rfloor}

\newcommand{\Hess}{\operatorname{Hess}}

\newcommand{\Ann}{\operatorname{Ann}}

\newcommand{\C}{{\mathbb C}}
\newcommand{\R}{{\mathbb R}}

\renewcommand{\P}{{\mathbb P}}

\theoremstyle{theorem}
\newtheorem{theorem}{Theorem}

\newtheorem{lemma}{Lemma}[section]
\newtheorem{corollary}{Corollary}[section]
\newtheorem{example}{Example}[section]

\newtheorem{remark}{Remark}[section]

\newtheorem{fact}{Fact}[section]

\begin{document}
	\title[]{A Hessian criterion for totally nonnegative Toeplitz matrices and a theorem of Cattani}
	\author{Chris McDaniel}
	\maketitle
	
	\begin{abstract}
		To a homogeneous real bivariate form we associate a family of Toeplitz matrices, as well as a family of auxiliary homogeneous real bivariate forms called higher Hessian polynomials.  We show that for a given form, certain positivity properties of its higher Hessian polynomials imply total nonnegativity of its Toeplitz matrices.  One of our main result turns out to be equivalent to a special case of a deep theorem from Hodge theory due to E. Cattani.  Our other result uses a recent theorem of S. Karp and K. Purbhoo related to problems in real Schubert calculus. 
	\end{abstract} 
	
	\section{Introduction}
	Let $F=F(X,Y)\in\R[X,Y]_d$ be a homogeneous real bivariate form, of some fixed degree $d$.  For each $i$ satisfying $0\leq i\leq d$, its $i^{th}$ Toeplitz matrix is denoted by $\phi^i_d(F)$, and its $i^{th}$ Hessian polynomial is denoted by $H^F_i=H^F_i(X,Y)$.  The Sperner number of $F$, denoted by $s(F)$, is defined to be the maximal value of the Hilbert function of a certain graded algebra associated to $F$ (\Cref{eq:SpernerDef}), although it can also be characterized as the maximum rank of the Toeplitz matrices $\phi^i_d(F)$ (\Cref{rem:Sperner3}), or equivalently, as the smallest index $i$ for which the Hessian polynomial $H_{i}^F$ is identically zero (\Cref{rem:Hess}).  We prove the following results.
	\begin{theorem}
		\label{thm:IntroCattani}
	For any $F\in\R[X,Y]_d$ with Sperner number $s=s(F)$, we have the following equivalences:
	\begin{enumerate}
		\item $\phi^{s-1}_d(F)$ is totally positive if and only if 
		$$H^F_i(X,Y)>0, \ \begin{cases} \forall (X,Y)\geq 0, \ (X,Y)\neq (0,0)\\ \forall 0\leq i\leq s-1\\ \end{cases}$$ 
		
		\item $\phi^{s-1}_d(F)$ is totally nonnegative if and only if $$H^F_i(X,Y)>0, \ \begin{cases} \forall (X,Y)>0\\ \forall 0\leq i\leq s-1\\ \end{cases}$$ 
	\end{enumerate}
	\end{theorem}
	
	\begin{theorem}
		\label{thm:IntroKarp}
	For any $F\in\R[X,Y]_d$ with Sperner number $s(F)$ and for any positive integer $r$ satisfying $1\leq r\leq s(F)$, we have the following implications:
	\begin{enumerate}
		\item $\phi^{r-1}_d(F)$ is totally positive if \begin{align*}
			H^F_i(X,Y)>0, & \ \ \begin{cases}\forall (X,Y)\geq 0, \ (X,Y)\neq (0,0)\\ \forall 0\leq i\leq r-1\\ \end{cases}, \ \text{and}\\
			H^F_{r-1}(X,1) & \ \ \text{has only real negative roots}.
		\end{align*}

		\item $\phi^{r-1}_d(F)$ is totally nonnegative if 
		\begin{align*}
			H^F_i(X,Y)>0, & \ \ \begin{cases}\forall (X,Y)>0\\ \forall 0\leq i\leq r-1\\ \end{cases}, \ \text{and}\\
			H^F_{r-1}(X,1) & \ \ \text{has only real nonpositive roots}.
		\end{align*}
		
	\end{enumerate}
	\end{theorem}
	\Cref{thm:IntroCattani} gives a necessary and sufficient condition for the total positivity/nonnegativity of the Toeplitz matrix $\phi^{s-1}_d(F)$, and the assumption that $s=s(F)$, the Sperner number of $F$, here is crucial.  \Cref{thm:IntroKarp} drops that assumption on the Sperner number, but it only gives a sufficient condition for total positivity/nonnegativity.  
	
	According to the theory of Macaulay duality, every homogeneous real form $F$ (in any number of variables) defines a graded oriented Artinian Gorenstein $\R$-algebra $A_F$.  The algebra $A_F$ is an algebraic model of the cohomology algebra (in even degrees) of a compact complex manifold, in the sense that it is a finite dimensional graded vector space that satisfies an analogue of Poincar\'e duality.  In the case of a K\"ahler manifold, its cohomology algebra is known to satisfy the ordinary Hodge-Riemann relations (ordinary HRR), which describes the signature of the multiplication maps between Poincar\'e dual graded components defined by powers of a single ``linear form'' in its K\"ahler cone.  Similarly, the mixed Hodge-Riemann relations (mixed HRR) describe the signature of multiplication by products of several, possibly distinct, K\"ahler forms.  The ordinary and mixed HRR conditions make sense for a general algebra $A_F$ and linear forms in an arbitrary subset $U$, and one can ask for which $F$ and $U$ do either or both of these conditions hold?  In general, the mixed HRR on an arbitrary subset $U$ implies the ordinary HRR on $U$, but not conversely.  On the other hand, a theorem of E. Cattani \cite{Cattani} implies that if $U$ is a convex cone, e.g. the K\"ahler cone, then the ordinary HRR on $U$ is equivalent to the mixed HRR on $U$.
	In a paper \cite{MMS,Errata} with P. Macias Marques and A. Seceleanu, we showed that in the two variable case, where $F=F(X,Y)$, \Cref{thm:IntroCattani} is equivalent to a special case of Cattani's theorem.  This is one way to prove \Cref{thm:IntroCattani}.  
	
	From another point of view, a recent result of S. Karp \cite[Theorem 1.1]{Karp} gives a Wronskian criterion for total positivity/nonnegativity in the real flag variety.  As we shall see (\Cref{fact:WronHess1}), the Wronskian and Hessian are closely related.  Taking this as our starting point, we show that the main ideas from Karp's proof can be applied to give another proof of \Cref{thm:IntroCattani}, and hence Cattani's theorem in that special two variable case.  Two key ingredients here are the Pl\"ucker expansion formula for the Hessian polynomial (\Cref{lem:Plucker}) and a topological criterion identifying the space of totally positive Toeplitz matrices as a union of connected components inside  a certain distinguished open set (\Cref{lem:ConnComp}).  The significance of the Sperner number is illustrated in \Cref{lem:S}.   
	
	In his paper \cite{Karp}, Karp also showed how his Wronskian criterion is related to some fundamental problems of real Schubert calculus, in particular, to the famous Shapiro-Shapiro conjecture, now a theorem of Mukhin, Tarasov and Varchenko \cite{MTV}.  In a recent preprint \cite{KP}, Karp and K. Purbhoo have given a proof of a totally positive version of the Mukhin-Tarasov-Varchenko theorem (\Cref{fact:KP}).  We use their result together with \Cref{fact:WronHess1} to prove an analogous result for Hessians (\Cref{cor:HessKarp}), and use this criterion to prove \Cref{thm:IntroKarp}.

	This paper is organized as follows.  In \Cref{sec:Prelim}, we define all prerequisite notions, and we prove \Cref{lem:ConnComp} and \Cref{lem:S} mentioned above.  In \Cref{sec:Proof}, we prove \Cref{thm:IntroCattani}.  In \Cref{sec:Schubert} we prove \Cref{cor:HessKarp} and \Cref{thm:IntroKarp}.  All requisite facts are stated without proofs for the sake of brevity, but precise references are given.  Examples are given throughout. 

	\section{Hodge Theory and Proof of \Cref{thm:IntroCattani}}
	\label{sec:Prelim}
	\subsection{Artinian Gorenstein algebras and Macaulay dual generators}
	Let $R=\R[x,y]$ denote the polynomial ring in two variables over the real numbers $\R$ equipped with the standard grading, i.e. $\deg(x)=\deg(y)=1$.  Let $R_d$ denote the $\R$-vector space of homogeneous forms of degree $d$.  Let $Q=\R[X,Y]$ be another copy of the standard graded polynomial ring in two variables (distinguished from $R$ by upper case letters), and let $R$ act on $Q$ by partial differentiation, i.e. $x\circ F=\frac{\partial F}{\partial X}$, $y\circ F=\frac{\partial F}{\partial Y}$.  Given $F\in Q_d$, its annihilator ideal is the (homogeneous) ideal 
	$$\Ann(F)=\left\{f\in R \ | \ f\circ F\equiv 0\right\}$$
	and its associated \emph{standard graded oriented Artinian Gorenstein algebra} is the quotient algebra 
	$$A=A_F=\frac{R}{\Ann(F)};$$
	it is a finite dimensional graded $\R$-vector space with graded components $A_i$, and socle degree $d$, and such that multiplication defines perfect pairings between complimentary graded components, called the \emph{intersection pairing},
	$$(-,-)_i\colon A_i\times A_{d-i}\rightarrow\R, \ (\alpha,\beta)_i=\left(\alpha\beta\right)\circ F.$$
	The polynomial $F$ is called the \emph{dual generator} of $A$.  We define its \emph{normalized coefficients} $(c_0,\ldots,c_d)$ by
		$$F=\sum_{k=0}^d\binom{d}{k}c_kX^kY^{d-k}.$$
	The Hilbert function of $A_F$ is the $(d+1)$-tuple of positive integers
	$$H(A_F)=(h_0,h_1,\ldots,h_d), \ h_i=\dim_\R\left(A_i\right).$$
	The \emph{Sperner number of $F$ (or $A_F$)} is defined to be the maximum value in the Hilbert function, i.e.
	\begin{equation}
		\label{eq:SpernerDef}
		s(F)=\max\left\{h_i \ | \ 0\leq i\leq d\right\}.
	\end{equation}
	The Hilbert function is determined by the socle degree $d=\deg(F)$ and the Sperner number $s=s(F)$, given by the formula: 
	\begin{equation}
		\label{eq:Hilbert}
		H(A_F)=(1,2,\ldots,s-1,s,s,\ldots,s,s,s-1,\ldots,2,1);
	\end{equation}
	the Sperner number is also characterized as the minimum degree of a generator in the annihilator ideal $\Ann(F)$; see e.g. \cite[Theorem 1.44]{IK}.  
	
	The next result states that polynomials in the same $\operatorname{GL}_2(\R)$-orbit define isomorphic algebras.  First some notation.  Let $A\in\operatorname{GL}_2(\R)$ be a $2\times 2$ real invertible matrix, and define the associated \emph{linear change of coordinates map} 
	$$\psi=\psi_A\colon Q\rightarrow Q, \ \psi_A(F)(X,Y)=F\left((X,Y)\cdot A\right).$$
	This also determines an \emph{adjoint map} 
	$$\psi^*=\psi^*_A\colon R\rightarrow R, \ \psi^*(f)(x,y)=f\left((x,y)\cdot A^T\right)$$
	where $A^T$ is the transpose of $A$.  Note that both $\phi$ and $\phi^*$ are algebra automorphisms.  
	
	The following is well known, e.g. \cite[Lemma 4.1]{HessWron}.
	\begin{fact}
		\label{lem:MDchange}
		For each $A\in\operatorname{GL}_2(\R)$, the maps $\psi=\psi_A\colon Q\rightarrow Q$ and $\psi^*=\psi^*_A\colon R\rightarrow R$ satisfy 
		\begin{equation}
			\label{eq:psiStar}
			\psi\left(\psi^*(f)\circ F\right)=f\circ \psi(F), \ \forall f\in R, \ \forall F\in Q.
		\end{equation}
		In particular, the map $\psi^*\colon R\rightarrow R$ passes to an isomorphism of the quotient algebras
		$$\overline{\psi^*}\colon \frac{R}{\Ann(\psi(F))}\rightarrow \frac{R}{\Ann(F)}.$$
	\end{fact}
	The next result seems to have been first proved by Iarrobino \cite[Theorem 2.9]{IMem} although it has been rediscovered and reproved many times over.
	\begin{fact}
		\label{fact:Tony}
		Given $A=A_F=\R[x,y]/\Ann(F)$ as above, there exists a non-empty Zariski open set of linear forms $\ell\in A_1$ for which the multiplication maps are isomorphisms:
		$$\times\ell^{d-2i}\colon A_i\rightarrow A_{d-i}, \ 0\leq i\leq \flo{d}.$$
	\end{fact}
	The property of \Cref{fact:Tony} that multiplication by powers of a linear form is an isomorphism is sometimes referred to as the \emph{strong Lefschetz property}, and it has been extensively studied over the past several decades; see \cite{HMMNWW} and the references therein.  We will discuss further properties and generalizations of these multiplication maps below. 
	\subsection{Hodge-Riemann relations}
	Let $A=A_F$ be as above.  Given a linear form $\ell\in A_1$, and an index $i$, $0\leq i\leq \flo{d}$, define its $i^{th}$ \emph{Lefschetz multiplication map} by 
	$$\times\ell^{d-2i}\colon A_i\rightarrow A_{d-i}$$
	and define its $i^{th}$ \emph{primitive subspace} as the kernel,
	$$P_{i,\ell}=\ker\{\times\ell^{d-2i+1}\colon A_i\rightarrow A_{d-i+1}\}.$$
	The primitive subspaces form a graded vector subspace of $A$ that lives in degrees $0\leq i\leq \flo{d}$; for degrees $i>\flo{d}$, we define $P_{i,\ell}$ to be the zero subspace.  Using the Lefschetz multiplication maps in conjunction with the intersection pairing on $A$, define the $i^{th}$ \emph{Lefschetz form} by 
	$$(-,-)_i^\ell\colon A_i\times A_i\rightarrow\R, \ (\alpha,\beta)_i^\ell=(\ell^{d-2i}\alpha \beta)\circ F;$$
	it is a symmetric bilinear form on the vector space $A_i$.
	We say that the pair $\left(A_F,\ell\right)$ satisfies the \emph{(ordinary) Hodge-Riemann relations up to degree $i$} (\emph{ordinary HRR$_i$}) if the $j^{th}$ Lefschetz form $(-,-)_j^\ell$ is $(-1)^j$-definite on the $j^{th}$ primitive subspace $P_{j,\ell}$ for all degrees $j$, $0\leq j\leq i$, i.e. 
	\begin{equation}
		\label{eq:OHR}
	(-1)^j\cdot (\alpha,\alpha)_j^\ell=(-1)^j\cdot \left(\ell^{d-2i} \alpha^2\right)\circ F>0, \ \begin{cases}
		\forall 0\neq \alpha\in P_{j,\ell}\\
		\forall 0\leq j\leq i\\
		\end{cases}
	\end{equation}	
	For any subset of nonzero linear forms $U\subset A_1$, we say that $A$ satisfies the \emph{ordinary Hodge-Riemann relations up to degree $i$ on $U$} (\emph{ordinary HRR$_i(U)$}) if $(A_F,\ell)$ satisfies ordinary HRR$_i$ for each $\ell\in U$.
	
	Given a $(d+1)$-tuple of linear forms $\mathcal{L}=(\ell_0,\ell_1,\ldots,\ell_d)\in\left(A_1\right)^{d+1}$, define its $i^{th}$ \emph{mixed Lefschetz multiplication maps} by 
	$$\times\ell_1\cdots\ell_{d-2i}\colon A_i\rightarrow A_{d-i}$$
	and define its $i^{th}$ \emph{mixed Lefschetz form}
	$$\left(-,-\right)^\mathcal{L}_i\colon A_i\times A_i\rightarrow \R, \ \left(\alpha,\beta\right)_i^\mathcal{L}=\left(\ell_1\cdots\ell_{d-2i}\alpha\beta\right)\circ F$$
	and its $i^{th}$ \emph{mixed primitive subspace}
	$$P_{i,\mathcal{L}}=\ker\{\times\ell_0\ell_1\cdots\ell_{d-2i}\colon A_i\rightarrow A_{d-i+1}\}.$$
	We say that the pair $\left(A_F,\mathcal{L}\right)$ satisfies the \emph{mixed Hodge-Riemann relations up to degree $i$} (\emph{mixed HRR$_i$}) if the $j^{th}$ mixed Lefschetz form $(-,-)^\mathcal{L}_j$ is non-degenerate on $A_j$ and $(-1)^j$-definite on the $j^{th}$ mixed primitive subspace $P_{j,\mathcal{L}}$ for all degrees $j$, $0\leq j\leq i$, i.e. 
	\begin{equation}
		\label{eq:MHR}
		(-1)^j\cdot (\alpha,\alpha)_j^\mathcal{L}=(-1)^j\cdot \left(\ell_1\cdots\ell_{d-2j}\alpha^2\right)\circ F>0, \ \begin{cases}\forall 0\neq \alpha\in P_{j,\mathcal{L}}\\
		\forall 0\leq j\leq i\\
		\end{cases}
	\end{equation}
		
	For any subset of nonzero linear forms, $U\subset A_1$, we say that $A_F$ satisfies the \emph{mixed Hodge-Riemann relations up to degree $i$ on $U$} (\emph{mixed HRR$_i(U)$}) if $(A_F,\mathcal{L})$ satisfies mixed HRR$_i$ for all $\mathcal{L}\in U^{d+1}$. 
	
	Note that the conditions in \Cref{eq:OHR,eq:MHR} are vacuous for values of $j$ in which the primitive subspace is trivial.  In general, if $A_F$ satisfies mixed HRR$_i(U)$, or more generally if the $j^{th}$ mixed Lefschetz maps from $U$ are isomorphisms for $0\leq j\leq i$, then for any $\mathcal{L}\in U$, its primitive subspaces must satisfy $$\dim_{\R}\left(P_{j,\mathcal{L}}\right)=h_j-h_{j-1};$$
	see \cite[Lemma 2.11, 2.12]{MMS}.  In particular, it follows from \Cref{eq:Hilbert} that the primitive subspaces can only live in degrees $i$ for $0\leq i\leq s-1\leq \flo{d}$ where $s$ is the Sperner number defined above.  Moreover, one can show that if $A_F$ satisfies ordinary or mixed HRR$_i(U)$ for $i\geq s-1$, then it must also satisfy HRR$_{\flo{d}}(U)$ by default, since for $j\geq s$ the primitive subspaces are trivial.  Finally note that, by definition, ordinary or mixed HRR$_i(U)$ always implies ordinary or mixed HRR$_{i-1}(U)$.  
	
	It is clear that for any subset $U$ and for any degree $i$, $0\leq i\leq \flo{d}$, mixed HRR$_i(U)$ implies ordinary HRR$_i(U)$.  Indeed, for any $\ell\in U$, take $\mathcal{L}=(\ell,\ell,\ldots,\ell)\in U^{d+1}$.  Then mixed HRR$_i$ for $(A_F,\mathcal{L})$ is equivalent to ordinary HRR$_i$ for $(A_F,\ell)$.  In general, the converse is not true however; see \Cref{ex:210} below.  On the other hand, we have the following beautiful result of E. Cattani \cite[Theorem 4.4]{Cattani}. 
	\begin{fact}[Cattani 2008]
		\label{thm:Cattani0}
		If $U$ is a convex cone, then $A_F$ satisfies ordinary HRR$_{\flo{d}}(U)$ if and only if it satisfies mixed HRR$_{\flo{d}}(U)$. 
	\end{fact}
	For simplicity, we shall restrict our notion of convex cone to mean either an \emph{open cone} defined as the positive span of finitely many vectors or a \emph{closed punctured cone} defined as the nonnegative span of finitely many vectors, excluding zero.
	We shall sometimes abuse notation slightly and refer to a convex cone of linear forms $U$ as a subset of $R_1$ and also its image $U$ in the quotient $A_1$.
	 
	The following example is \cite[Example 2.10]{MMS}, included here for completeness; it shows that the condition that $U$ is a convex cone in \Cref{thm:Cattani0} cannot be dropped.
	\begin{example}
		\label{ex:210}
		Let $F=X^3+Y^3$, with oriented Artinian Gorenstein algebra 
		$$A=A_F=\frac{\R[x,y]}{\Ann(X^3+Y^3)}=\frac{\R[x,y]}{(xy,x^3-y^3)}$$
		of socle degree $d=3$ and Hilbert function $H(A)=(1,2,2,1)$.
		Given a general linear form $\ell=ax+by$, we compute its primitive subspaces in degrees $i=0$ and $i=1=\flo{d}$:
		\begin{align*}
			P_{0,\ell}= & \ker\{\times\ell^4\colon A_0\rightarrow A_4=0\}= \langle 1\rangle\\
			P_{1,\ell}= & \ker\{\times\ell^{2}\colon A_1\rightarrow A_3\}=  \langle\alpha=b^2x-a^2y\rangle
		\end{align*}
		Then the ordinary HRR give the inequalities
		\begin{align*}
			\langle 1,1\rangle^\ell_0= & \ell^3\circ F=6(a^3+b^3)> 0\\
			(-1)\cdot \langle \alpha,\alpha\rangle^\ell_1= & \ell\alpha^2\circ F=-6ab(a^3+b^3)>0.
		\end{align*}
		Therefore we see that $A$ satisfies ordinary HRR$_{\flo{d}}(U)$ where $U$ is the disjoint union of two open convex cones
		\begin{align*}
			\mathcal{C}_1= & \left\{ax+by \ | \ 0<-a<b\right\}\\
			\mathcal{C}_2= & \left\{ax+by \ | \ 0<-b<a\right\}.
			\end{align*}
		On the other hand, for mixed HRR, given four general linear forms $\ell_i=a_ix+b_iy$, $0\leq i\leq 3$, with $\mathcal{L}=(\ell_0,\ell_1,\ell_2,\ell_3)\in R_1^{4}$, its primitive subspace in degree $i=1$ is
		$$P_{1,\mathcal{L}}=\ker\{\times\ell_0\ell_1\colon A_1\rightarrow A_0\}=\langle \alpha= b_0b_1x-a_0a_1y\rangle$$
		and the mixed HRR give the inequalities
		\begin{align*}
			\langle 1,1\rangle^{\mathcal{L}}_0= & \ell_1\ell_2\ell_3\circ F= 6(a_1a_2a_3+b_1b_2b_3)>0\\
			(-1)\cdot\langle \alpha,\alpha\rangle^{\mathcal{L}}_1= & -\ell_1\alpha^2\circ F=-6a_1b_1(b^2_0b_1+a_0^2a_1)>0
		\end{align*} 
		Note that $A$ satisfies mixed HRR$_{\flo{d}}(\mathcal{C}_j)$ for each $j=1,2$, individually.  On the other hand, taking for instance, $\ell_1=\ell_2=\ell_3=-x+2y\in\mathcal{C}_1$ and  $\ell_0=2x-y\in\mathcal{C}_2$, shows that $A$ does not satisfy mixed HRR$_{\flo{d}}(U)$.		
	\end{example}
	
	We will also see in \Cref{ex:Cattanii} that \Cref{thm:Cattani0} does not hold in general if we replace HRR$_{\flo{d}}$ with HRR$_i$, for $i<\flo{d}$.  In fact, to get from ordinary HRR$_i(U)$ to mixed HRR$_i(U)$ for $i<\flo{d}$, we need another condition, as in \Cref{thm:IntroKarp}; see \Cref{rem:Cattani}.


	\begin{remark}
		\label{rem:Cattani0}
		We remark that Cattani's original statement in \cite[Theorem 4.4]{Cattani} is in the much more general setting of ``Hodge-Lefschetz modules'', which includes, in particular, Artinian Gorenstein algebras in an arbitrary number of variables, and even more general modules over such algebras.  His proof relies on deep results from the theory of variations of Hodge structure, in particular the ``Descent Lemma'' of Cattani, Kaplan and Schmid \cite[Lemma 1.16]{CKS}. 
	\end{remark}
	     
	The following is an immediate consequence of \Cref{lem:MDchange}; for details see \cite[Corollary 4.13]{MMS}.  
	\begin{fact}
		\label{lem:U1U2}
		Let $\psi\colon Q\rightarrow Q$ be a linear change of coordinates map, with adjoint map $\psi^*\colon R\rightarrow R$ as in \Cref{lem:MDchange}.  For any subsets $U_1,U_2\subset R_1$, if $A_F$ satisfies the ordinary, respectively mixed, HRR$_i(U_1)$ and $\psi^*(U_2)\subset U_1$, then, $A_{\psi(F)}$ must satisfy the ordinary, respectively mixed, HRR$_i(U_2)$.
	\end{fact}
	
	In light of \Cref{lem:U1U2}, and since any two (open or closed punctured) convex cones in the plane can be mapped bijectively onto one another by an invertible linear transformation, it suffices to prove \Cref{thm:Cattani0} for one fixed convex cone $U$.
	For the remainder of this paper, we fix $U=\mathcal{K}_{>0}\subset R_1$ to be the open cone of strictly positive linear forms
	\begin{equation}
		\label{eq:K}
	\mathcal{K}_{>0}=\left\{ax+by \ | \ a,b>0\right\}\cong \R^2_{>0}
\end{equation}
	and its punctured closure $\mathcal{K}_{\geq 0}^*\subset R_1$ is the cone of nonnegative and nonzero linear forms
	\begin{equation}
		\label{eq:Kstar}
		\mathcal{K}_{\geq 0}^*=\left\{ax+by \ | \ a,b\geq 0, \ (a,b)\neq (0,0)\right\}\cong \R^2_{\geq 0}\setminus\{(0,0)\}.
		\end{equation} 
	
	\subsection{Higher Hessians}
	Given $F$ as above, and $i$ satisfying $0\leq i\leq \flo{d}$, take $\mathcal{E}=\bigsqcup_{i\geq 0}\mathcal{E}_i$, where $\mathcal{E}_i$ is the standard ordered monomial basis of $R_i$, $\mathcal{E}_i=\{e_p^i=x^py^{i-p} \ | \ 0\leq p\leq i\}$, and define the \emph{$i^{th}$ Hessian matrix of $F$ (with respect to $\mathcal{E}$)} to be the $(i+1)\times (i+1)$ polynomial matrix 
	\begin{align}
		\label{eq:Hessian}
		\Hess_i(F)=  \Hess_i(F,\mathcal{E})= & \left((e_p^i\cdot e^i_q)\circ F\right)_{0\leq p,q\leq i}\\
		\nonumber= & \left(\frac{\partial^{2i}F}{\partial X^{p+q}\partial Y^{2i-p-q}}\right)_{0\leq p,q\leq i}.
	\end{align} 
	The entries of $\Hess_i(F)$ are homogeneous polynomials of degree $d-2i$, and for any real values of the variables $X$ and $Y$, it is a real symmetric matrix.  Note that for $i=0$, the Hessian is just $F$ itself.  Also, note that for $0\leq i\leq s(F)-1$, $\Ann(F)_i=\{0\}$ and hence $\mathcal{E}_i$ passes to a basis for the quotient algebra $A_F$.	The Hessian matrix of $F$ has a natural interpretation in terms of the Lefschetz multiplication maps in $A_F$.  The following is well known, e.g. \cite[Lemma 3.4]{MMS}.
	\begin{fact}
		\label{lem:HessMat}
		Given a homogeneous form $F\in Q_d$ and its associated oriented Artinian Gorenstein algebra $A=A_F$ with Sperner number $s=s(F)$, let $\ell=ax+by\in R_1$ be any linear form and fix $i$, $0\leq i\leq s-1$.  Then we have 
		$$\Hess_i(F,\mathcal{E})|_{(a,b)}=\frac{1}{(d-2i)!}\left((e^i_p,e^i_q)_i^\ell\right)_{0\leq p,q\leq \flo{d}}.$$ 
		In other words, the $i^{th}$ Hessian matrix, evaluated at $X=a$ and $Y=b$, is equal to the matrix for the $i^{th}$ Lefschetz form with respect to the basis $\mathcal{E}_i\subset A_i$, or equivalently, it is the matrix of the $i^{th}$ Lefschetz multiplication map
		$$\times\ell^{d-2i}\colon A_i\rightarrow A_{d-i}$$
		with respect to the basis $\mathcal{E}_i$ and its dual basis $\mathcal{E}_i^*$.  
	\end{fact}
	
	Using \Cref{eq:Hessian}, we can also define the $i^{th}$ Hessian matrix of $F$ with respect to any other homogeneous basis $\mathcal{F}=\sqcup_{i\geq 0}\mathcal{F}_i$,  where $\mathcal{F}_i=\{f^i_p \ | \ 0\leq p\leq i\}\subset R_i$.  In this case the two matrices are related by the usual change of basis formula:
	\begin{equation}
		\label{eq:change}
		\operatorname{Hess}_i(F,\mathcal{F})=M^T\cdot\operatorname{Hess}_i(F,\mathcal{E})\cdot M
	\end{equation}
	where $M=(m_{pq})_{0\leq p,q\leq i}$ is the invertible (real constant) matrix where $f^i_q=\sum_{p=0}^im_{pq}e^i_p$, and $M^T$ is its transpose.  As a special case, consider a linear change of coordinates $\psi\colon Q\rightarrow Q$ with its adjoint map $\psi^*\colon R\rightarrow R$, and let $\psi\left(\Hess_i(F,\mathcal{E})\right)$ be the matrix whose entries are $\psi$ applied to the entries of $\Hess_i(F,\mathcal{E})$.  Then according to \Cref{eq:psiStar}, we have 
	\begin{equation}
		\label{eq:HessianCOC}
		\psi\left(\Hess_i(F,\psi^*\left(\mathcal{E}\right))\right)=\Hess_i(\psi(F),\mathcal{E}).
	\end{equation}

	Define the \emph{$i^{th}$ Hessian polynomial of $F$} by 
\begin{equation}
	\label{eq:HessPoly}
	H^F_i=H^F_i(X,Y)=(-1)^{\flo{i+1}}\det\left(\Hess_i(F)\right).
\end{equation}
	Of course one could also define the $i^{th}$ Hessian polynomial with respect to any other homogeneous basis $\mathcal{F}$ by 
	$$H^{F,\mathcal{F}}_i(X,Y)=(-1)^{\flo{i+1}}\cdot \det\left(\Hess_i(F,\mathcal{F})\right),$$
	and according to \Cref{eq:change}, $H^F_i$ and $H^{F,\mathcal{F}}_i$ would only differ by a \emph{positive} constant multiple.  In particular, given a linear change of coordinates $\psi\colon Q\rightarrow Q$ with adjoint map $\psi^*\colon R\rightarrow R$, \Cref{eq:HessianCOC} implies 
	\begin{equation}
		\label{eq:HessDetCOC}
		\psi\left(H^F_i(X,Y)\right)=c_i(\psi)^2\cdot H^{\psi(F)}_i(X,Y)
	\end{equation}
	where $c_i(\psi)$ is the determinant of the matrix for restriction of $(\psi^*)^{-1}$ to $R_i$ with respect to the basis $\mathcal{E}_i$.

	The strange sign on the right hand side of \Cref{eq:HessPoly} is justified by the following result; see \cite[Lemma 3.11]{MMS}:
	\begin{fact}
		\label{lem:HRR}
		Fix a homogeneous form $F\in Q_d$ of degree $d$ with Sperner number $s=s(F)$, fix $i$, $0\leq i\leq \flo{d}$ and let $\mathcal{K}_{>0}$ and $\mathcal{K}_{\geq 0}^*$ be as in \Cref{eq:K,eq:Kstar}.  Then 
		\begin{enumerate}
		\item $A_F$ satisfies ordinary HRR$_i\left(\mathcal{K}_{\geq 0}^*\right)$ if and only if 
		$$H_j^F(X,Y)>0, \ \begin{cases} \forall (X,Y)\geq 0, \ (X,Y)\neq (0,0)\\
			\forall 0\leq j\leq \min\{i,s-1\}\\ \end{cases}$$
			
		\item $A_F$ satisfies ordinary HRR$_i\left(\mathcal{K}_{>0}\right)$ if and only if 
		$$H_j^F(X,Y)>0, \ \begin{cases} \forall (X,Y)>0\\
			\forall 0\leq j\leq \min\{i,s-1\}.\\ \end{cases}$$
		\end{enumerate}
	\end{fact}

	\begin{remark}
		\label{rem:Hess}
		Note that for $i\geq s$, the polynomial $H_i^F$ must be identically zero, due to the fact that for $i\geq s$, the monomials $\mathcal{E}_i=\left\{x^py^{i-p} \ | \ 0\leq p\leq i\right\}$ satisfy a dependence relation in the quotient $\left(A_F\right)_i$.  Moreover, note that for $0\leq i\leq s-1$, $H^F_i$ is not identically zero by \Cref{fact:Tony} and \Cref{lem:HessMat}.  In particular, this gives an alternative characterization of the Sperner number:
		$$s(F)=\min\{i \ | \ H^F_i\equiv 0\}.$$

	\end{remark}
	
	\subsection{Toeplitz Matrices}
	Given a homogeneous $d$-form $F\in Q_d$, write $$F=\sum_{k=0}^d\binom{d}{k}c_kX^kY^{d-k},$$
	and define the \emph{$i^{th}$ Toeplitz matrix of $F$}, $\phi^i(F)$, to be the $(i+1)\times (d-i+1)$ Toeplitz matrix formed by the normalized coefficient sequence of $F$, i.e. 
	\begin{equation}
		\label{eq:phi}
		\phi^i_d(F)=\left(\begin{array}{cccc} c_i & c_{i+1} & \cdots & c_d\\
		\vdots & \ddots & \ddots & \vdots\\
		c_0 & c_1 & \cdots & c_{d-i}\\ \end{array}\right)=\left(c_{i+q-p}\right)_{\substack{0\leq p\leq i\\ 0\leq q\leq d-i\\}}.
	\end{equation}
	According to \cite[Lemma 4.5]{MMS}, for $0\leq i\leq \flo{d}$, the rank of $\phi^i_d(F)$ satisfies   
	\begin{equation}
		\label{eq:rank}
		\operatorname{rank}\left(\phi^i_d(F)\right)=\min\{i+1,s(F)\}=\dim_\R\left(\left(A_F\right)_i\right).
	\end{equation}
	The matrix $\phi^i_d(F)$ is closely related to the catalecticant matrix of $F$ in \cite[Definition 1.2]{IK}.
	\begin{remark}
		\label{rem:Sperner3}
		Note that \Cref{eq:rank} gives yet another characterization of the Sperner number:
		$$s(F)=\max\left\{\operatorname{rank}\left(\phi^i_d(F)\right) \ | \ 0\leq i\leq \flo{d}\right\}.$$
	\end{remark}

	Recall that a matrix is totally nonnegative, respectively totally positive, if all of its minor determinants are nonnegative, respectively positive.  
	The following result was proved in \cite[Theorems 1,2]{MMS,Errata}:
	\begin{fact}
		\label{lem:mixedHRR}
		Fix a homogeneous form $F\in Q_d$ of degree $d$ and Sperner number $s=s(F)$, and fix $i$, $0\leq i\leq \flo{d}$.  Then 
		\begin{enumerate}
			\item $A_F$ satisfies mixed HRR$_i\left(\mathcal{K}_{\geq 0}^*\right)$ if and only if $\phi^i_d(F)$ is totally positive.
			
			\item $A_F$ satisfies mixed HRR$_i\left(\mathcal{K}_{>0}\right)$ if and only if $\phi^i_d(F)$ is totally nonnegative.
		\end{enumerate}
	\end{fact}
	
	In the language of \cite{MMS,Errata}, polynomials $F$ satisfying \Cref{lem:mixedHRR}(1) are called strictly $i$-Lorentzian and their limits, called $i$-Lorentzian polynomials, are characterized by \Cref{lem:mixedHRR}(2).  
	
	\begin{remark}
		\label{rem:Cattani}
		\begin{enumerate}
			\item From \Cref{lem:HRR} and \Cref{lem:mixedHRR}, we obtain the following rephrasing of \Cref{thm:IntroCattani}:  
			
			``For $U=\mathcal{K}_{\geq 0}^*$ or $U=\mathcal{K}_{>0}$, $A_F$ satisfies ordinary HRR$_{\flo{d}}(U)$ if and only if it satisfies mixed HRR$_{\flo{d}}(U)$.''
			
			From this reformulation, and \Cref{lem:U1U2}, it follows that \Cref{thm:IntroCattani} is equivalent to \Cref{thm:Cattani0}.
			
			\item A similar rephrasing of \Cref{thm:IntroKarp} is 
			
			``For $U=\mathcal{K}_{\geq 0}^*$ or $U=\mathcal{K}_{>0}$, and for fixed $i$, $0\leq i<s-1$, if $A_F$ satisfies ordinary HRR$_{i}(U)$ and the $i^{th}$ Hessian polynomial $H_i^F$ has all real negative roots, then $A_F$ must also satisfy mixed HRR$_i(U)$.''
			\Cref{ex:Cattanii} below shows that in general, one should not expect ordinary HRR$_i(U)$ to imply mixed HRR$_i(U)$ for $i<\flo{d}$ without some additional hypotheses, such as real rooted-ness of the Hessian polynomial. 
		\end{enumerate}     
	\end{remark} 
	
	\begin{example}
		\label{ex:Cattanii}
		Let 
		$$F=Y^4+12XY^3+12X^2Y^2+8X^3Y+X^4.$$
		Then its Sperner number $s(F)=3=\flo{d}+1$, and its $i=1$ Hessian polynomial is 
		$$H^F_1=144(7Y^4 + 8XY^3 -X^2Y^2 + 2X^3Y + 2X^4).$$
		Since $F=H^F_0>0$ and $H^F_1>0$ on $\left(\R^2_{\geq 0}\right)^*$, it follows from \Cref{lem:HRR} that the oriented Artinian Gorenstein algebra $A_F$ satisfies ordinary HRR$_1\left(\mathcal{K}_{\geq 0}^*\right)$.  On the other hand, note that the Toeplitz matrix 
		$$\phi^1_4(F)=\left(\begin{array}{cccc} 3 & 2 & 2 & 1\\ 1 & 3 & 2 & 2\\ \end{array}\right)$$
		is not totally positive, and hence, by \Cref{lem:mixedHRR}, $A_F$ does not satisfy mixed HRR$_1\left(\mathcal{K}_{\geq 0}^*\right)$.  The problem here is evidently that $s-1=2$ and the $2^{nd}$ Toeplitz matrix for $F$
		$$\phi^2_4(F)=\left(\begin{array}{ccc} 2 & 2 & 1\\ 3 & 2 & 2\\ 
		1 & 3 & 2\\ \end{array}\right)$$
		has determinant $\det(\phi^2_4(F))=-5$, and hence is not totally positive.  It should also be noted that the $1^{st}$ Hessian polynomial equation 
		$$H_1^F(X,1)=2X^4+2X^3-X^2+8X+7=0$$
		has only two real solutions, not four.		
	\end{example}
			
	
	\subsection{Pl\"ucker expansion of the Hessian}
	The following is a useful formula that expresses the Hessian polynomials in terms of the Pl\"ucker coordinates of the matrix $\phi^i_d(F)$.  First we need some notation.  Let $A\in\mathcal{M}(m,n)$ ($=$ the set of real $m\times n$ matrices) for some $1\leq m\leq n$, and for any $1\leq k\leq m$ and any $k$-subsets $I,J\in\binom{[n]}{m}$, let $A_{IJ}$ denote the $k\times k$ submatrix whose rows and columns are the rows and columns of $A$ indexed by $I$ and $J$, respectively, and let 
	$$\Delta_{IJ}(A)=\det(A_{IJ})$$
	denote the corresponding minor.  If $k=m$, then $I=[m]$ and we write this maximal (initial) minor as $\Delta_J(A)$.  These maximal minors $$\left\{\Delta_J(A) \ | \ J\in\binom{[n]}{m}\right\}$$ 
	are called the \emph{Pl\"ucker coordinates of the matrix $A$}, since they uniquely determine the point in projective space representing the equivalence class of $A$ as a point $[A]\in\operatorname{Gr}_m(\R^n)$ in the Grassmannian of $m$-planes in $\R^n$; see \cite[Section 2]{Karp} and references therein for further details.
	
	Given an $m$-subset $J=\{1\leq i_1<\cdots<i_m\leq n\}$, let $\lambda(J)=(\lambda_1\geq \cdots\geq \lambda_m)$ denote the partition with parts defined by 
	$$\lambda_{k}=i_{m-k+1}-(m-k+1).$$
	Note that $\lambda(J)$ is a partition with part sizes at most $n-m$ and number of (nonzero) parts at most $m$; in other words $\lambda(J)$ is a partition whose Young diagram fits inside a rectangle of size $m\times (n-m)$.  In fact, it is a basic combinatorial fact that $\lambda$ defines a bijective map from the set of $m$-subsets of $[n]$ to the set of partitions of length $m$ and parts of size at most $n-m$, i.e. partitions that fit inside the $m\times (n-m)$ rectangle.  Let
	$\lambda'(J)=(\lambda_1'\geq\cdots\lambda_{n-m}')$ denote the conjugate partition, with parts defined by 
	$$\lambda'_{k}=\# \{\lambda_i\geq k\}.$$
	Then $\lambda'(J)$ fits inside the conjugate $(n-m)\times m$ rectangle.
	Define the number  
	$$N'_J=\prod_{1\leq i<j\leq n-m}\frac{\lambda_i'-\lambda_j'- i+j}{j-i};$$
	Then for each $J\in\binom{[n]}{m}$, $N'_J$ is a positive integer which counts the number of semi-standard Young tableaux of shape $\lambda'(J)$ using the numbers in $[n-m]$.  Let $D_i=(i+1)(d-2i)$, and denote the size of the partition by $$|\lambda(J)|=\lambda_1+\cdots+\lambda_m=\lambda_1'+\cdots+\lambda_{n-m}'=|\lambda'(J)|.$$
	The following can be deduced from \cite[Lemma 4.19]{MMS}; see also \cite[Proposition 3.3]{HessWron}.	
	\begin{fact}
		\label{lem:Plucker}
		For any homogeneous $d$-form $F\in Q_d$ with Sperner number $s=s(F)$, and for any $i$ satisfying $0\leq i\leq s-1$, the $i^{th}$ Hessian polynomial of $F$ satisfies
		$$H^F_i(X,Y)=\left(\frac{d!}{(d-2i)!}\right)^{i+1}\cdot \sum_{J\in\binom{[d-i+1]}{i+1}}N'_J\cdot \Delta_J\left(\phi^i_d(F)\right)\cdot X^{|\lambda(J)|}\cdot Y^{D_i-|\lambda(J)|}.$$ 
	\end{fact} 
	
	\begin{remark}
		\label{rem:Wronskian}
		Earlier incarnations of \Cref{lem:Plucker} have appeared in the work of I. Gessel \cite[Theorem 16]{Gessel}.		
		A formula for the Wronskian, analogous to the one in \Cref{lem:Plucker}, has appeared in the paper of K. Purbhoo \cite[Proposition 2.3]{Purbhoo}; see also \cite[Proposition 2.11]{Karp}.  We will see (\Cref{fact:WronHess1}) that the Hessian is really a special case of the Wronskian, as first pointed out by Pasch \cite{Pasch} in 1874; for further details and proofs, we refer the reader to \cite{HessWron}.  See \Cref{ex:H3} for an example using \Cref{lem:Plucker}.  
	\end{remark}

	\subsection{Topology of Toeplitz matrices}
	Let $\mathcal{T}(m,n)$ denote the set of $m\times n$ Toeplitz matrices, as a linear subspace of the Euclidean space $\mathcal{M}(m,n)$, the space of $m\times n$ real matrices.  
	Then \Cref{eq:phi} defines a linear isomorphism 
	$$\phi^i_d\colon Q_d\rightarrow \mathcal{T}(i+1,d-i+1).$$
	Endowing $Q_d\cong\R^{d+1}$ with the Euclidean topology makes $\phi^i_d$ into a homeomorphism of topological spaces.
	
	Denote by $\mathcal{T}(m,n)^{>0}$ the open subset of totally positive $m\times n$ Toeplitz matrices and by $\mathcal{T}(m,n)^{\geq 0}$ the closed subset of totally nonnegative ones.  
	The next result is \cite[Theorem 4.21]{MMS,Errata}, and we refer the reader there for its proof.
	\begin{fact}
		\label{fact:Closure}
		The closure of the set of totally positive Toeplitz matrices is the set of totally nonnegative Toeplitz matrices, i.e.
		$$\overline{\mathcal{T}(m,n)^{>0}}=\mathcal{T}(m,n)^{\geq 0}$$
	\end{fact}
	
	We shall make use of the following fundamental property of Toeplitz matrices, which is a direct consequence \Cref{lem:mixedHRR}.
	\begin{fact}
		\label{lem:Inherit}
		For any homogeneous $d$-form $F\in Q_d$, and for any nonnegative integer $i$, with $0\leq i\leq \flo{d}$, if $\phi^{i}_d(F)$ is totally positive, respectively nonnegative, then $\phi^j_d(F)$ is also totally positive, respectively nonnegative, for every $j$, with $0\leq j\leq i$.
	\end{fact}
	
	Given $A\in\mathcal{M}(m,n)$, and $1\leq k\leq \min\{m,n\}$ define a \emph{corner minor of size $k$} to be the determinant of any $k\times k$ submatrix formed by either the first $k$-rows and the last $k$-columnn or the the last $k$-rows and the first $k$-columns.
   	Define a \emph{consecutive minor of size $k$} to be the determinant of a $k\times k$ submatrix formed by some consecutive subsets of rows and columns of $A$, i.e. 
   	$$\Delta_{IJ}(A), \ \text{where} \ \begin{cases} I=\{i,i+1,\ldots,i+k-1\} & \text{for some} \ i\\ J=\{j,j+1,\ldots,j+k-1\} & \text{for some} \ j\\ \end{cases}.$$
   	Finally, we define a \emph{maximal minor} of $A$ to be the determinant of any $k\times k$ submatrix where $k=\min\{m,n\}$.
   	
   	\begin{remark}
   		\label{rem:Errata}
   		The totally positive case of \Cref{lem:Inherit} can be deduced directly from properties of Toeplitz matrices, together with the following well known criterion of Fekete, e.g. \cite[Corollary 3.1.5]{FJ}:
   		
   		``A matrix is totally positive if and only if every consecutive minor is positive.''
   		
   		\noindent The totally nonnegative case of \Cref{lem:Inherit} can then be deduced directly from \Cref{fact:Closure} using a straightforward limiting argument.
   		
   		In our original paper \cite{MMS}, we had tacitly assumed \Cref{lem:Inherit} and used it to prove \Cref{lem:mixedHRR} and from that also \Cref{fact:Closure}, resulting in a rather unfortunate circular argument.  This circular argument was repaired in \cite{Errata}, where we give yet another proof of \Cref{lem:Inherit} that is independent of both \Cref{lem:mixedHRR} and \Cref{fact:Closure}. 
   	\end{remark}
   	
   	We shall use the following fact, which is \cite[Theorem 3.1.12]{FJ}, and we refer the reader there for its proof.
   	\begin{fact}
   		\label{fact:FJ}
   		A matrix $A\in\mathcal{M}(m,n)$ is totally positive if and only if $A$ is totally nonnegative and all of its corner minors and all of its consecutive maximal minors are nonzero (positive).
   	\end{fact}
   	\Cref{fact:FJ} has the following topological interpretation.  For positive integers $m,n$, define the open subset $\mathcal{O}(m,n)\subset\mathcal{T}(m,n)$ consisting of $m\times n$ Toeplitz matrices, all of whose corner minors of size $k\leq \min\{m,n\}$ are nonzero and all of whose consecutive maximal minors are nonzero.  
   	\begin{lemma}
   		\label{lem:ConnComp}
   		For any positive integers $m,n$, we have 
   		$$\mathcal{T}(m,n)^{>0}=\mathcal{T}(m,n)^{\geq 0}\cap \mathcal{O}(m,n).$$
   		In particular, $\mathcal{T}(m,n)^{>0}$ is a union of connected components of the subspace $\mathcal{O}(m,n)$.
   	\end{lemma}
   	\begin{proof}
   		For the first statement, the containment $\subseteq$ is clear and the reverse containment is a direct consequence of \Cref{fact:FJ}.  For the last statement, note that since $\mathcal{T}(m,n)^{>0}$ is open in the ambient space $\mathcal{M}(m,n)$, it is also open in the subspace $\mathcal{O}(m,n)$.  Moreover, since $\mathcal{T}(m,n)^{\geq 0}$ is closed in $\mathcal{M}(m,n)$, it follows that from the equality that $\mathcal{T}(m,n)^{>0}$ is also closed in the subspace $\mathcal{O}(m,n)$.  Therefore $\mathcal{T}(m,n)^{>0}$ is an open and closed subset of $\mathcal{O}(m,n)$, and hence the result follows from elementary point-set topology.   
   	\end{proof}
   	
   	\begin{remark}
   		\label{rem:Lusztig}
   		\Cref{lem:ConnComp} plays a similar role in our proof of \Cref{thm:IntroCattani}, as an analogous theorem of Lusztig \cite[Proposition 8.14]{Lusztig} plays in Karp's proof of his Wronskian criterion for total positivity in the flag variety; see \cite[Proposition 2.6 and Proof of Theorem 1.1(ii)]{Karp}.   
   	\end{remark}
   	
   	
   	The next result will help us determine when the Toeplitz matrix $\phi^{s-1}_d(F)$ lies in the distinguished open set $\mathcal{O}(s,n)$.  This is where we need the assumption that $s$ is equal to the Sperner number of $F$.
   	\begin{lemma}
   	\label{lem:S}
   	Let $F\in Q_d$ be a homogeneous $d$-form of Sperner number $s=s(F)$.  If $A_F$ satisfies ordinary HRR$_{s-1}\left(\mathcal{K}_{\geq 0}^*\right)$, then for every $r$, $0\leq r\leq d-2(s-1)$, the multiplication map 
   	$$\times y^{d-r-2(s-1)}\colon \left(A_{x^r\circ F}\right)_{s-1}\rightarrow\left(A_{x^r\circ F}\right)_{d-r-(s-1)}$$
   	is an isomorphism. In particular, the consecutive maximal minors of $\phi^{s-1}_d(F)$ are all nonzero. 
   	\end{lemma}
   	\begin{proof}
   	Since $s=s(F)$ is the Sperner number of $F$, we have $\dim_\R\left(A_F\right)_{j}=s$ for all degrees $j$ satisfying $s-1\leq j\leq d-s+1$, and hence since the multiplication map 
   	$$\times y^{d-2(s-1)}\colon \left(A_F\right)_{s-1}\rightarrow\left(A_F\right)_{d-s+1}$$
   	is an isomorphism (by HRR$_{\flo{d}}\left(\mathcal{K}_{\geq 0}^*\right)$, since $y\in\mathcal{K}_{\geq 0}^*$), it follows that the multiplication maps 
   	$$\times y\colon\left(A_F\right)_j\rightarrow\left(A_F\right)_{j+1}$$
   	must also be isomorphisms for all $s-1\leq j\leq d-s$.  Next, if $$\alpha\in\ker\left\{\times y^{d-r-2(s-1)}\colon\left(A_{x^r\circ F}\right)_{s-1}\rightarrow\left(A_{x^r\circ F}\right)_{d-r-(s-1)}\right\}$$
   	then it follows that $x^r\cdot \alpha$ must be in the kernel of the injective multiplication map
   	$$\times y^{d-r-2(s-1)}\colon\left(A_F\right)_{s+r-1}\rightarrow\left(A_F\right)_{d-r-(s-1)}.$$
   	Hence $x^r\alpha=0$ in $A_F$, but since $x\in \mathcal{K}_{\geq 0}^*$ too, it follows that $\alpha=0$ in $A_F$, and hence the map above is injective as claimed.  For the last assertion, note that the $(s-1)^{st}$ Hessian polynomial for $x^r\circ F$ evaluated at $(0,1)$ satisfies 
   	$$H^{x^r\circ F}_{s-1}(0,1)=\lambda\cdot \det\left(\begin{array}{ccc} c_{s-1+r} & \cdots & c_{2s-2+r}\\
   		\vdots & \ddots & \vdots\\
   		c_r & \cdots & c_{s-1+r}\\ \end{array}\right)$$
   	where $\lambda$ is a positive nonzero constant and the determinant is the $(r-1)^{st}$ $s\times s$ maximal consecutive minor of $\phi^{s-1}_d(F)$, which, according to \Cref{lem:HessMat}, is nonzero, by the first statement.
   	\end{proof}
   	
   	\subsection{Proof of \Cref{thm:IntroCattani}}
   	\label{sec:Proof}
   	We are now in a position to prove \Cref{thm:IntroCattani}.  

\begin{proof}[Proof of \Cref{thm:IntroCattani}]
	First we prove (1).  Assume $\phi^{s-1}_d(F)$ is totally positive.  Then, according to \Cref{lem:Inherit}, so is $\phi^i_d(F)$ for each $0\leq i\leq s-1$.  It follows from the Pl\"ucker formula in \Cref{lem:Plucker}, that for each $0\leq i\leq s-1$, the $i^{th}$ Hessian $H^F_i(X,Y)$ has positive coefficients, and therefore we have 
	$$H^F_i(X,Y)>0, \ \begin{cases}(X,Y)\geq 0, \ (X,Y)\neq (0,0)\\ \forall 0\leq i\leq s-1\\ \end{cases}.$$ 
	For the converse implication, assume the Hesssian criterion holds.  Then in particular, for each $0\leq i\leq s-1$, the leading coefficients of the $i^{th}$ Hessian polynomial must be positive.  According to our Pl\"ucker formula, the leading coefficients of $H^F_i(X,Y)$ (corresponding to the empty partition and the full rectangular partition) are exactly the maximal corner minors of $\phi^i_d(F)$, which, in turn, are exactly the corner minors of $\phi^{s-1}_d(F)$ of size $i+1$.  Moreover, according to \Cref{lem:HRR} our Hessian criterion is equivalent to the graded oriented Artinian Gorenstein algebra $A_F$ satisfying ordinary HRR$_{s-1}\left(\mathcal{K}_{\geq 0}^*\right)$.  Hence according to \Cref{lem:S}, it follows that the maximal consecutive minors of $\phi^{s-1}_d(F)$ are also nonzero.  In particular, we have shown that our Hessian criterion on $F$ implies that its Toeplitz matrix lies in our ``distinguished open subspace'', i.e. $\phi^{s-1}_d(F)\in\mathcal{O}(s,d-s+2)$.  It remains to show that $\phi^{s-1}_d(F)$ lies in the totally positive component.  Following an idea of Karp, define the one-parameter family of linear change of coordinates
	$$\psi_t\colon Q\rightarrow Q, \ \psi_t(F)(X,Y)=F(X+tY,Y), \ t\in\R.$$
	Then for each $t\in\R$, we have an algebra isomorphism
	$$\overline{\psi_t^*}\colon A_{\psi_t(F)}\rightarrow A_{F}.$$
	In particular, the Sperner number $s(\psi_t(F))=s(F)$ is independent of $t$, and also, for each $t\geq 0$, $\psi_t^*\left(\mathcal{K}_{\geq 0}^*\right)\subset\mathcal{K}_{\geq 0}^*$.  It follows from \Cref{lem:U1U2} that since $A_F$ satisfies the ordinary HRR$_{s-1}\left(\mathcal{K}_{\geq 0}^*\right)$, $A_{\psi_t(F)}$ also satisfies ordinary HRR$_{s-1}\left(\mathcal{K}_{\geq 0}^*\right)$.  Hence, by our previous arguments, we must have 
	$$\phi^{s-1}_d\left(\psi_t(F)\right)\in\mathcal{O}(s,d-s+2), \ \forall t\geq 0.$$
	Furthermore, note that for $t>>0$ sufficiently large, every linear form in the cone $$\psi_t\left(\mathcal{K}_{\geq 0}^*\right)=\left\{ax+(b+t)y \ | \ a,b\geq 0, \ (a,b)\neq (0,0)\right\}$$
	can be made arbitrarily close to the line spanned by the linear form $\ell=y$.  Therefore, by a straightforward deformation argument, e.g. \cite[Corollary 4.15]{MMS}, we can deduce that for $t>>0$ sufficiently large, the algebra $A_{\psi_t(F)}$ satisfies mixed HRR$_{s-1}\left(\mathcal{K}_{\geq 0}^*\right)$, which is equivalent to $\phi_d^{s-1}\left(\psi_t(F)\right)\in\mathcal{T}(s,d-s+2)^{>0}$ by \Cref{lem:mixedHRR}.  Thus, since $t\mapsto\phi^{s-1}_d\left(\psi_t(F)\right)$ defines a continuous path in the space of Toeplitz matrices which, for $t\geq 0$, lies in our distinguished open subspace $\mathcal{O}(s,d-s+2)$, and for $t>>0$, lies in the totally positive component $\mathcal{T}(s,d-s+2)^{>0}$, it follows from \Cref{lem:ConnComp} that $\phi^{s-1}_d(F)=\phi^{s-1}_d(\psi_0(F))$ must lie in the same connected component, and hence $\phi^{s-1}_d(F)\in\mathcal{T}(s,d-s+2)^{>0}$ too.
	
	For (2), assume first that $\phi^{s-1}_d(F)$ is totally nonnegative.  Then by \Cref{lem:Inherit}, $\phi^i_d(F)$ is totally nonnegative for each $0\leq i\leq s-1$, and it follows from our Pl\"ucker formula that our Hessian criterion
	$$H^F_i(X,Y)>0, \ \begin{cases}\forall (X,Y)>0\\ \forall 0\leq i\leq s-1\\ \end{cases}$$
	is satisfied.  Conversely, assuming the above Hessian criterion holds, it follows from \Cref{lem:HRR} that the algebra $A_F$ satisfies the ordinary HRR$_{s-1}\left(\mathcal{K}_{>0}\right)$.  Consider the one-parameter family of linear change of coordinates
	$$\theta_t\colon Q\rightarrow Q, \ \theta_t(F)(X,Y)=F(X+tY,tX+Y), \ t\in\R, \ t\neq -1,1$$
	and the corresponding algebra isomorphisms
	$$\overline{\theta_t^*}\colon A_{\theta_t(F)}\rightarrow A_F, \ t\in\R, t\neq -1,1.$$
	Since, for each $0< t<1$, $\theta_t^*\left(\mathcal{K}_{\geq 0}^*\right)\subset \mathcal{K}_{>0}$, it follows from \Cref{lem:U1U2} that $A_{\theta_t^*(F)}$ satisfies ordinary HRR$_{s-1}\left(\mathcal{K}_{\geq 0}^*\right)$, and therefore, by our previous argument, that $\phi^{s-1}_d(\theta_t(F))$ is totally positive.  Since $\phi^{s-1}_d(F)=\lim_{t\to 0}\phi^{s-1}_d(\theta_t(F))$, it follows that $\phi^{s-1}_d(F)$ must be totally nonnegative.
\end{proof}
	\begin{remark}
		\label{rem:Karp}
		The method of proof employed in the proof of \Cref{thm:IntroCattani} above, is essentially the same method that Karp used to establish his Wronskian criterion for totally positive flags \cite[Theorem 1.1]{Karp}.  In fact, it is a  remarkable coincidence that the argument above applies almost verbatim if we replace ``Hessian'' with ``Wronskian'' and ``Toeplitz matrix'' with ``flag''.  The connection between ``Hessian'' and ``Wronskian'' will be made clear below (\Cref{fact:WronHess1}), but the connection between ``Toeplitz matrix'' and ``flag'' is still somewhat mysterious.
	\end{remark} 
   	
   		  
   
   \section{Wronskians and Proof of \Cref{thm:IntroKarp}}	
   \label{sec:Schubert}
   \subsection{Wronskians and Totally Positive Grassmannians}
   In this section, it will be useful to work over the complex numbers.  Fix positive integers $r,s$ satisfying $r\leq s+1$, and consider a $r$-dimensional subspace $V\subset\C[X,Y]_s$ in the $(s+1)$-dimensional vector space of homogeneous $s$-forms.  It will be convenient to think of $V\in \operatorname{Gr}_r(\C[X,Y]_s)$ as a point in the Grassmannian of $r$-planes in the $(s+1)$-dimensional ambient space of homogeneous $s$-forms.  We say that $V$ is real if it has a basis of real homogeneous $s$-forms; in this case we shall write $V\in\operatorname{Gr}_r(\R[X,Y]_s)\subset\operatorname{Gr}_r(\C[X,Y]_s)$.  
   Given an arbitrary subspace $V\in\operatorname{Gr}_r(\C[X,Y]_s)$ and a basis $F_1,\ldots,F_r\in V$, where $F_i=\sum_{j=0}^sa_{ij}X^jY^{s-j}$ for some complex numbers $a_{ij}\in\C$, we can represent $V$ by the $r\times (s+1)$ complex matrix 
   $$M_V=\left(a_{ij}\right)_{\substack{1\leq i\leq r\\ 0\leq j\leq s\\}}.$$
   Then $M_V$ is uniquely determined up to a left $\operatorname{GL}_r(\C)$-action, and $V$ is real if and only if it has a matrix representative $M_V$ whose entries are real numbers. 
   For each $r$-subset $I=\{0\leq i_1<\cdots<i_r\leq s\}\subset\binom{[s]_0}{r}$, denote by $\Delta_I(M_V)$ the determinant of the $r\times r$ submatrix formed from the $r$-columns of $M_V$ indexed by $I$.  The set of all such maximal minor determinants of $M_V$ are called the Pl\"ucker coordinates of $V$ in $\operatorname{Gr}_r(\C[X,Y]_s)$; by multiplicativity of determinants, they are independent of the choice matrix representative, and they uniquely determine $V$ as a point in the Grassmannian under its Pl\"ucker embedding.  We say that a real subspace $V\in\operatorname{Gr}_r(\R[X,Y]_s)$ is a point of the totally positive, respectively nonnegative, Grassmannian if all of its Pl\"ucker coordinates can be taken to be positive, respectively nonnegative; in that case, we write
   $$V\in\operatorname{Gr}_r(\R[X,Y]_s)^{>0}, \ \text{respectively} \ V\in\operatorname{Gr}_r(\R[X,Y]_s)^{\geq 0}.$$
   Given a subspace $V\in\operatorname{Gr}_r\left(\C[X,Y]_s\right)$ and a basis $\{F_1,\ldots,F_r\}\subset V$, we define its \emph{homogeneous Wronskian} by the formula
   $$W(F_1,\ldots,F_r;X,Y)=\frac{1}{Y^{\frac{r(r-1)}{2}}}\cdot \det\left(\left(\frac{\partial^{i-1}F_j}{\partial X^{i-1}}\right)_{1\leq i,j\leq r}\right).$$
   One can show that $W(F_1,\ldots,F_r;X,Y)$ is indeed a nonzero homogeneous polynomial of degree $N\coloneqq r(s+1-r)$, and it is independent of the choice of basis for $V$ up to scaling by a nonzero constant, e.g. \cite[Lemma 2.2]{HessWron}; sometimes as a shorthand we write $W(V;X,Y)$.  Note that the specialization $Y\mapsto 1$ yields the familiar non-homogeneous Wronskian of the de-homogenized forms, i.e.
   $$W(F_1,\ldots,F_r;X,1)=\operatorname{Wron}(F_1(X,1),\ldots,F_r(X,1)).$$
   One can further show that the homogeneous Wronskian gives a well-defined map 
   $$[W]\colon\operatorname{Gr}_r\left(\C[X,Y]_s\right)\rightarrow\P\left(\C[X,Y]_N\right)$$
   which associates to every subspace $V\in\operatorname{Gr}_r(\C[X,Y]_s)$ a well defined projective class of homogeneous $N$-form, denoted by $\left[W(V;X,Y)\right]\in\P\left(\C[X,Y]_N\right)$; in fact, according to Purbhoo \cite[Theorem 1.1]{Purbhoo}, this is a finite flat morphism of schemes.
   
   In \cite[Theorem 1.3, Proposition 1.12]{KP}, S. Karp and K. Purbhoo have obtained the following Wronskian criterion for membership in the totally positive/nonnegative Grassmannian.
   \begin{fact}
   	\label{fact:KP}
   	Let $V\in\operatorname{Gr}_r(\C[X,Y]_s)$ be any $r$-dimensional subspace of $s$-forms, and let $W(V;X,Y)$ be any representative of the projective class of its homogeneous Wronskian.  Then we have the following implications.
   	\begin{enumerate}
   		\item If $W(V;X,1)$ has only real and negative roots, then $V$ is totally positive, i.e. $$V\in\operatorname{Gr}_r(\R[X,Y]_s)^{>0}.$$
   		\item If $W(V;X,1)$ has only real and nonpositive roots, then $V$ is totally nonnegative, i.e. $$V\in\operatorname{Gr}_r(\R[X,Y]_s)^{\geq 0}.$$
   	\end{enumerate}
   \end{fact}
   
   \subsection{Wronskians versus Hessians}
   Set $d=r+s-1$.  Fix a real homogeneous $d$-form $F\in\R[X,Y]_d$ with Sperner number $s(F)$ satisfying $r\leq s(F)$, and define the $r$-dimensional real subspace
   $V^F_r\in\operatorname{Gr}_r(\R[X,Y]_s)$ by
   $$V^F_r=\operatorname{span}_\R\left\{\left.F_i=\frac{\partial^{r-1}F}{\partial X^{r-i}\partial Y^{i-1}}\right| 1\leq i\leq r\right\}$$
   (note that the condition $r\leq s(F)$ implies $V^F_r$ is $r$-dimensional).   
   Write $W(V^F_r;X,Y)$ as a shorthand for the homogeneous Wronskian $W(F_1,\ldots,F_r;X,Y)$.  The following result, which was already known to Pasch in 1874 \cite{Pasch}, says that the homogeneous Wronskian of $V_r^F$ and the $(r-1)^{st}$ Hessian of $F$ are projectively equivalent.  For a proof, the reader is referred to Pasch's original paper \cite[page 179- 180]{Pasch}, or, for a combinatorial interpretation, see \cite[Proposition 3.4]{HessWron}.
   \begin{fact}
   	\label{fact:WronHess1}
   	Define 
   	$$c(r,s)=\prod_{i=0}^{r-1}\frac{(s-r+1)!}{(s-i)!}.$$
   	Then we have 
   	$$W(V^F_r;X,Y)=c(r,s)\cdot H^F_{r-1}(X,Y).$$
   \end{fact}
   
   \begin{corollary}
   	\label{cor:HessKarp}
   	For any homogeneous $d$-form $F\in Q_d$ of Sperner number $s(F)$, and for any $r$ satisfying $r\leq s(F)$, we have the following implications:
   	\begin{enumerate}
   		\item If $H^F_{r-1}(X,1)$ has only real and negative roots, then 
   		$$\phi^{r-1}_d(F)\in\operatorname{Gr}_r\left(\R[X,Y]_{d-r+1}\right)^{>0}.$$ 
   		\item If $H^F_{r-1}(X,1)$ has only real and nonpositive roots, then 
   		$$\phi^{r-1}_d(F)\in\operatorname{Gr}_r\left(\R[X,Y]_{d-r+1}\right)^{\geq 0}.$$
   	\end{enumerate}
   \end{corollary}
   \begin{proof}
   	For (1), note that if $H^F_{r-1}(X,1)$ has all real negative zeros, then, according to \Cref{fact:WronHess1}, so does $W(V_r^F;X,Y)$, and hence by \Cref{fact:KP}, we have $$V_r^F\in\operatorname{Gr}_r\left(\R[X,Y]_{d-r+1}\right)^{>0}.$$  
   	If $F=\sum_{k=0}^d\binom{d}{k}c_kX^kY^{d-k}$, then we have 
   	$$F_{i}=\frac{\partial^{r-1}F}{\partial X^{r-i}\partial Y^{i-1}}=(d)_{r-1}\cdot\sum_{k=0}^{d-r+1}\binom{d-r+1}{k}c_{k+r-i}X^kY^{d-r+1-k}$$
   	where $(d)_{r-1}=d(d-1)\cdots (d-(r-1)+1)$, and the corresponding matrix representative for $V=V^F_{r}$ is the product of the Toeplitz matrix $\phi^{r-1}_d(F)$ and a diagonal binomial matrix, i.e.
   	\begin{align*}
   		M_V= & (d)_{r-1}\cdot \left(\binom{d-r+1}{k}c_{k+r-i}\right)_{\begin{array}{l}1\leq j\leq r\\ 0\leq k\leq d-r+1\\ \end{array}}\\
   		= & \phi^{r-1}_d(F)\cdot\left(\delta_{k,m}\cdot \binom{d-r+1}{k}\cdot(d)_{r-1}\right)_{0\leq k,m\leq d-r+1}.
   	\end{align*}
   	Since $V\in\operatorname{Gr}_r(\R[X,Y]_{d-r+1})^{>0}$, it follows that the maximal minors of $M_V$ must all be non-zero and have the same sign, and hence, the same must hold for the maximal minors of $\phi^{r-1}_d(F)$.  The argument for (2) is similar. 
   \end{proof}
   It is important to note that \Cref{cor:HessKarp} is only a statement about the maximal minors of a matrix; in particular, it does not assert that the matrix is totally positive or nonnegative.
   \begin{example}
   	\label{ex:H3}
   	Take $F=2X^3+3X^2Y-Y^3$ so that 
   	$$\Hess_1(F)=\left(\begin{array}{rr} 12X+6Y & 6X\\ 
   		6X & -6Y\\ \end{array}\right)$$
   	and 
   	$$H^F_1=36(X+Y)^2.$$
   	Note that for $r=s=2$, we have
   	$$W(V_2^F;X,Y)=\frac{1}{Y}\det\left(\left(\begin{array}{ll} 6X^2+6XY & 3X^2-3Y^2\\
   	12X+6Y & 6X\\ \end{array}\right)\right)=18(X+Y)^2$$
   	and hence 
   	$$W(V_2^F;X,Y)=\frac{1!}{2!}\cdot\frac{1!}{1!}\cdot H^F_1(X,Y)$$
   	as predicted by \Cref{fact:WronHess1}.
   	Note that 
   	$$\phi^1_3(F)=\left(\begin{array}{rrr} 0 & 1 & 2\\ -1 & 0 & 1\\ \end{array}\right)\in\operatorname{Gr}_2\left(\R[X,Y]_3\right)^{>0};$$
   	in particular, its maximal minors are all positive as predicted by \Cref{cor:HessKarp}.  Note however that $\phi^1_3(F)$ is not even totally nonnegative.  We can also confirm our Pl\"ucker expansion formula in \Cref{lem:Plucker} for $d=3$ $i=1$:
   	\begin{align*}
   		H^F_1(X,Y)= & \left(\frac{3!}{1!}\right)^2\cdot \sum_{J\in\binom{[3]}{2}}N'_J\cdot\Delta_J(\phi^1_3(F))\cdot X^{|\lambda(J)|}\cdot Y^{2-|\lambda(J)|}\\ 
   		= & 36\cdot \left(1\cdot \left|\begin{array}{rr} 0 & 1\\ -1 & 0\\ \end{array}\right|\cdot Y^2+1\cdot \left|\begin{array}{rr} 0 & 2\\ -1 & 3\\ \end{array}\right|\cdot XY+1\cdot\left|\begin{array}{rr} 1 & 2\\ 0 & 1\\ \end{array}\right|\cdot X^2\right)\\
   		= & 36(X+Y)^2.
   	\end{align*}
   \end{example}
   
   \subsection{Proof of \Cref{thm:IntroKarp}}
   We are now in a position to prove \Cref{thm:IntroKarp}.
   \begin{proof}[Proof of \Cref{thm:IntroKarp}]
   We prove (1).  Our strategy is the same as in the proof of \Cref{thm:IntroCattani}: we first show that our Hessian condition implies the Toeplitz matrix $\phi^{r-1}_d(F)$ lies in our distinguished open set, then that our one-parameter family preserves our Hessian criterion along the nonnegative real axis, and finally for large nonnegative values of the parameter it lands in the totally positive component.  Assume that our Hessian criterion is satisfied, that is, for some $1\leq r\leq s(F)$,
   \begin{align*}
   	H^F_i(X,Y)>0, & \ \ \begin{cases}\forall (X,Y)\geq 0, \ (X,Y)\neq (0,0)\\ \forall 0\leq i\leq r-1\\ \end{cases}, \ \text{and}\\
   	H^F_{r-1}(X,1) & \ \ \text{has only real negative roots}.
   	\end{align*}
   	As above, it follows from our Pl\"ucker formula in \Cref{lem:Plucker} that the leading coefficients of $H^F_i(X,Y)$ for $0\leq i\leq r-1$, which are the corner minors of $\phi^{r-1}_d(F)$, are all positive.  Moreover, since we are assuming $H^F_{r-1}(X,1)$ has only real negative zeros, and that the maximal corner minors of $\phi^{r-1}_d(F)$ are all positive, it follows from \Cref{cor:HessKarp} that all maximal minors of $\phi^{r-1}_d(F)$ must be positive too.  Hence it follows that $\phi^{r-1}_d(F)$ does indeed lie in our distinguished open set, i.e. $\phi^{r-1}_d(F)\in \mathcal{O}(r,d-r+2)$.
   	Next, as in the proof of \Cref{thm:IntroCattani}, consider the one parameter family $\psi_t\colon Q\rightarrow Q, \ \psi_t(F)(X,Y)=F(X+tY,Y)$.  Again, since $A_F$ satisfies ordinary HRR$_{r-1}\left(\mathcal{K}_{\geq 0}^*\right)$, by \Cref{lem:HRR}, and since  $\psi_t^*\left(\mathcal{K}_{\geq 0}^*\right)\subset\mathcal{K}_{\geq 0}^*$ for all $t\geq 0$, it follows from \Cref{lem:U1U2} that $A_{\psi_t(F)}$ also satisfies ordinary HRR$_{r-1}\left(\mathcal{K}_{\geq 0}^*\right)$ for all $t\geq 0$, and hence also satisfies the first part of the Hessian condition.  To see that it also satisfies the second part, we appeal to \Cref{eq:HessDetCOC}, which implies that the univariate polynomial $H^{\psi_t(F)}_{r-1}(X,1)$ has the same roots as $H^F_{r-1}(X+t,1)$.  In particular, if $t\geq 0$, it follows that if $H^F_{r-1}(X,1)$ has all real negative roots, then so does $H^{\psi_t(F)}_{r-1}(X,1)$.  This implies that $$\phi^{r-1}_d\left(\psi_t(F)\right)\in \mathcal{O}(r,d-r+2), \ \forall t\geq 0.$$
   	As above, a straightforward deformation argument implies that for $t>>0$ sufficiently large, $\phi^{r-1}_d\left(\psi_t(F)\right)$ actually lies in the totally positive component, hence by the connected component property in \Cref{lem:ConnComp}, it follows that $\phi^{r-1}_d\left(\psi_t(F)\right)$ must be totally positive for all $t\geq 0$, which proves (1).   
   	
   	To prove (2), we use a limiting argument, again, as in the proof of \Cref{thm:IntroCattani}.  Assume that $F$ satisfies the Hessian condition
   	\begin{align*}
   		H^F_i(X,Y)>0, & \ \begin{cases}\forall (X,Y)>0\\
   		\forall 0\leq i\leq r-1\\ \end{cases}, \ \text{and}\\
   		H^F_{r-1}(X,1) & \ \text{has only real nonpositive roots}.\\ 
   	\end{align*}
   Then for $0<t<1$, the one-parameter family of linear change of coordinates maps $\theta_t\colon Q\rightarrow Q$, $\theta_t(F)(X,Y)=F(X+tY,tX+Y)$ satisfies $\theta_t^*\left(\mathcal{K}_{\geq 0}^*\right)\subset \mathcal{K}_{>0}$, and hence 
   $$H^{\theta_t(F)}_i(X,Y)>0, \ \begin{cases}\forall (X,Y)\geq 0, \ (X,Y)\neq (0,0)\\ \forall 0\leq i\leq r-1\\ \end{cases}$$
   for all $0<t<1$ by \Cref{lem:HRR}.  We claim that for $0<t<1$, $H^{\theta_t(F)}_{r-1}(X,1)$ also has all real negative roots.  More generally, we claim that if $G(X,Y)$ is a homogeneous $d$-form such that $G(X,1)$ has only real nonpositive roots, then the form $G_t(X,Y)=\theta_t(G)(X,Y)=G(X+tY,tX+Y)$ has the property that $G_t(X,1)$ has only real negative roots for all $0<t<1$.  Indeed note that 
   $$G_t(X,1)=(tX+1)^d\cdot G\left(\frac{X+t}{tX+1},1\right)$$
   and hence, if $r_1,\ldots,r_m\in\R$ are the nonpositive roots of $G(X,1)$, then, for $0< t< 1$, $r_1(t),\ldots,r_m(t)\in\R$ are the strictly negative roots of $G_t(X,1)$, where  
   $$r_i(t)=\frac{r_i-t}{1-tr_i}, \ 1\leq i\leq m$$
   (note if $0<t$ and $-\infty<r_i\leq 0$ then $1-tr_i>0$ and $r_i-t<0$.)
   It follows then from our previous argument that $\phi^{r-1}_d\left(\theta_t(F)\right)$ is totally positive for all $0<t<1$, and hence, since $\phi^{r-1}_d(F)=\lim_{t\to 0}\phi^{r-1}_d\left(\theta_t(F)\right)$, $\phi^{r-1}_d(F)$ must be totally nonnegative, as desired. 
   \end{proof}

\end{document}